\documentclass[12pt]{article}
\usepackage{amsmath,amsthm,amssymb}
\usepackage{times}
\usepackage{enumerate}
\usepackage{amsfonts}
\usepackage{fancyhdr}
\usepackage{titlesec}
\usepackage{cite}
\usepackage{ifthen}
\usepackage{amssymb}
\usepackage{fancyhdr}
\usepackage{titlesec}
\usepackage[arrow,matrix]{xy}
\usepackage{pifont}
\usepackage{stmaryrd}
\usepackage{setspace}
\usepackage{indentfirst}
\usepackage{amsmath,amssymb,amscd,bbm,amsthm,mathrsfs,dsfont}
\theoremstyle{definition}
\newtheorem{theorem}{Theorem}[section]
\newtheorem{corollary}[theorem]{Corollary}
\newtheorem{lemma}[theorem]{Lemma}

\newtheorem{remark}[theorem]{Remark}

\numberwithin{equation}{section}

\newcommand{\subjclass}[1]{\bigskip\noindent\emph{2010 Mathematics Subject Classification:}\enspace#1}
\newcommand{\keywords}[1]{\noindent\emph{Keywords:}\enspace#1}

\begin{document}


\baselineskip=16pt


\title{On the  Cauchy problem of  a  sixth-order   Cahn-Hilliard equation  arising in oil-water-surfactant mixtures  }

\author{Xiaopeng Zhao\\
\small{College of Sciences, Northeastern University, Shenyang 110004, China }\\
\small{E-mail address: zhaoxiaopeng@mail.neu.edu.cn}
 }

\date{}

\maketitle


\begin{abstract}We study the global well-posedness and asymptotic behavior of solutions for the Cauchy problem of three-dimensional sixth order  Cahn-Hilliard equation arising in oil-water-surfactant mixtures. First, by using the pure energy method and a standard continuity argument, we prove that there exists a unique global strong solution provided that the $H^2$-norm of initial data is sufficiently small. Moreover, we also establish the suitable negative Sobolev norm estimates and obtain the time decay rate of strong solutions. It is worth pointing out that  although   the problem we considered is a sixth-order parabolic equation,  the time decay rate  is equivalent to the decay rate of fourth-order generalized heat equation, which is better than our expect.


\end{abstract}

{\small\small
 \section{Introduction}

In 1990s, Gompper et. al.\cite{Gompper1,Gompper2} introduced the following free energy
functional
\begin{equation}
\label{1-a}
\mathcal{F}\{u\}=\int G(u,\nabla u,\Delta u)dx,
\end{equation}
with the density given by
$$
G(u,\nabla u,\Delta u)=f(u)+\frac12a(u)|\nabla u|^2+\frac{\delta}2|\Delta u|^2,
$$
to describe the dynamics of phase transitions in ternary oil-water-surfactant systems, where
  $u(x,t)$ describes the scalar order parameter which is proportional to the local difference between oil and water concentrations, $\delta$ denotes the mobility and the second gradient energy coefficient,  $a(u)$ is the first gradient energy coefficient which may be of arbitrary sign,  and $f(u)$ denotes the multiwell volumetric free energy density \cite{PI,vis,vis2}, respectively.

Let $M$ be the mobility and $\mu$   the chemical potential difference between the oil and water phases. Applying mass conservation, i.e.
$$\frac{\partial u}{\partial t}=-\hbox{div} j,$$
with the mass flux $j$ given by
$$
j=-M\nabla \mu.
$$
Moreover, the chemical potential can be defined by the constitutive equation
$$
\mu=\frac{\delta \mathcal{F}\{u\}}{\delta u},
$$
where $\frac{\delta \mathcal{F}\{u\}}{\delta u}$ is the first variation of the function $\mathcal{F}\{u\}$. Let the mobility $M\equiv\hbox{const}$, 
we end up with the following sixth order viscous Cahn-Hilliard type equation
\begin{equation} \label{1-b}
\left\{ \begin{aligned}
        &u_t=M\Delta\mu, \\
                  &\mu=\delta\Delta^2u-a(u)\Delta u-\frac12a'(u)|\nabla u|^2+f(u).
                          \end{aligned} \right.
                          \end{equation}

Many papers have studied the initial boundary value problem of equation (\ref{1-b}) from the point of view of global well-posedness.
In \cite{PI},
 Pawlow and Zajaczkowski assumed that the considered space $\Omega\subset\mathbb{R}^3$ is a bounded domain with a boundary  of class $C^6$, the free energy density $f(u)$ is a sixth order polynomial,  $a(u)=g_0+g_2u^2$ with $g_0\in\mathbb{R}$ and $g_2>0$,
 proved the existence of unique global smooth solution which depends continuously on the initial datum;
Moreover, by using  the B\"{a}cklund transformation and the Leray-Schauder fixed point theorem, Pawlow and Zajaczkowski\cite{PI2} proved the global unique solvability of equation (\ref{1-b})   in the Sobolev space $H^{6,1}(\Omega\times(0,T))$ under the assumption that the initial datum is in $H^3(\Omega)$.
    Schimperna and Pawlow \cite{vis2}  discussed the existence, uniqueness and parabolic regularization of a weak solution to the initial boundary value problem of equation (\ref{1-b}) with a singular (e.g., logarithmic) character. The authors also supposed that the parameter $\delta=0$, invistigated a fourth order system, considered the existence of weak solutions under very general conditions by means of a fixed point argument. In 2013,  Schimperna and Pawlow \cite{Sch} continued to study the fourth order system with singular diffusion. The authors proved that, for any final time $T$, the system admits a unique energy type weak solution, and for any $\tau>0$, such solution is classical. Moreover, based on Leray-Schauder's fixed point theorem and Campanato spaces, Liu and Wang \cite{Liu1} prove the existence of time-periodic solutions for the initial-boundary value problem of  equation (\ref{1-b}) in two space dimensions.

  There's also some results on the  equation (\ref{1-b}) with viscous term.   In the paper of  Pawlow and Zajaczkowski \cite{vis}, applying Leray-Schauder fixed point theorem and suitable estimates, the authors established the existence and uniqueness of a global in time regular solution. Very recently, based on the estimates of  weighted Sobolev spaces, Duan and Zhao \cite{DZ} considered the existence of global attractor for equation (\ref{1-b}) with viscous term in a 2D belt unbounded domain.

A Cauchy problem in mathematics asks for the solution of a partial differential equation that satisfies certain conditions that are given on a hypersurface in the domain.
It is worth pointing out that   the study of Cauchy problem on higher order nonlinear diffusion equations    is also interesting. There are many classical results related to this topic  (e.g., Caffarelli and Muler\cite{Caffarelli}, Dlotko, Kania and Sun\cite{Dlotko1},   Cholewa and Rodriguez-Bernal\cite{Cholewa1}  for the global existence of higher order diffusion equations; Dlotko and Sun\cite{Dlotko2}, Savostianov and Zelik\cite{Savo} for the global dynamics of higher order diffusion equations). As far as we know, there's no reference concerning this aspect of the sixth-order  Cahn-Hilliard equation till now. Can we establish some well-posedness results for the Cauchy problem of sixth-order Cahn-Hilliard equation (\ref{1-b})? We will answer this question in this paper.

Consider the Cauchy problem of sixth-order viscous Cahn-Hilliard equation (\ref{1-b})  in $\mathbb{R}^3$. The problem is state as follows:
\begin{equation} \label{1-1f}
\left\{ \begin{aligned}
        &u_t -\delta\Delta^3u=-\Delta\left[a(u)\Delta u+\frac12a'(u)|\nabla u|^2-f(u)\right], \quad x\in\mathbb{R}^3,\\
                  &u(x,0)=  u_0(x),
                          \end{aligned} \right.
                          \end{equation}
\begin{remark}There are some different choices on the functions $f(s)$ and $a(s)$. For example: on the basis of Landau-Ginzburg free energy,  Gompper et al.\cite{Gompper1,Gompper2} choose $f(s)$ and $a(s)$ as
 \begin{equation}\label{1-c}
f(s)=(s+1)^2(s^2+h_0)(s-1)^2,\quad a(s)=g_0+g_2s^2,
\end{equation}
where $h_0$, $g_0$ and $g_2>0$ are constants. Based on the fourth order gradient free energy, Pawlow and Zajaczkowski\cite{PI2} suppose that
 \begin{equation}\label{1-d}
f(s)=(1-\alpha)\frac{s^2}2+\frac{s^4}4,\quad a(s)=-2.
\end{equation}
A specific free energy with composition-dependent gradient energy coefficient $a(u)$ also arise in the modeling of phase separation in polymers\cite{23,vis}. This energy, known as Flory-Huggins-de Gennes energy, has the form (\ref{1-a}) with $\delta=0$,
$$F(s)=\int_0^sf(\tau)d\tau=
(1-s)\log(1-s)+(1+s)\log(1+s)-\frac{\lambda}2s^2,\quad\lambda\geq0,
$$ and the singular coefficient
$$
a(s)=\frac1{(1-s)(1+s)}.$$
\end{remark}
\begin{remark}
The sixth order Cahn-Hilliard equation model is only a phenomenological model. Hence, various modifications of it, for example, \cite{AM1,AM2,Ko,Ko1} and the reference therein,   has been proposed in order to capture the dynamical picture of the phase transition phenomena better.
\end{remark}

In this paper, we suppose that $f(s)$ and $a(s)$ satisfy (\ref{1-c}). There are two cases to be considered: First, the parameter $g_0>0$, equation (\ref{1-1f})$_1$ can be rewritten as
 \begin{equation} \label{1-1q1}
 u_t -\delta\Delta^3u+g_0\Delta^2u=-\Delta\left[g_2u^2\Delta u+g_2u|\nabla u|^2-(u-1)^2(u+1)^2(u^2+h_0)\right] .
                          \end{equation}
Second, if the parameter $g_0\leq 0$, we would like to rewrite problem (\ref{1-1f}) as
 \begin{equation} \label{1-1q2}
\begin{aligned}
       u_t -\delta\Delta^3u+(1-g_0)\Delta^2u=&-\Delta\left[ g_2\left(u+\sqrt{\frac{1-2g_0}{g_2}}\right)\left(u-\sqrt{\frac{1-2g_0}{g_2}}\right)\Delta u\right.
        \\
        &\left. +g_2u|\nabla u|^2-(u-1)^2(u+1)^2(u^2+h_0)\right]. \\
                  \end{aligned}
                          \end{equation}
For simplicity, we consider the following Cauchy problem in $\mathbb{R}^3$:
 \begin{equation} \label{1-1}
\left\{ \begin{aligned}
        &u_t -\delta\Delta^3u+\kappa_0\Delta^2u=-\Delta\left[ \kappa_1\left(u+ \kappa_2\right)\left(u-\kappa_2\right)\Delta u\right.
        \\
        &\left.\quad\quad\quad\quad\quad\quad\quad\quad\quad\quad\!\!+\kappa_1u|\nabla u|^2-(u-1)^2(u+1)^2(u^2+h_0)\right], \\
                  &u(x,0)=  u_0(x),
                          \end{aligned} \right.
                          \end{equation}
where $\delta>0$, $\kappa_0>0$, $\kappa_1>0$, $\kappa_2\geq0$ and $h_0\in\mathbb{R}$ are   constants. If $\kappa_0=g_0>0$, $\kappa_1=g_2$ and $\kappa_2=0$, then equation (\ref{1-1})$_1$ is equivalent to (\ref{1-1q1}). Moreover, if $\kappa_0=1-g_0>0$, $\kappa_1=g_2$ and $\kappa_2=\sqrt{\frac{1-2g_0}{g_2}}$, we obtain equation (\ref{1-1q2}).

\begin{remark}
Since we consider  problem (\ref{1-1}) in $\mathbb{R}^3$, the  Laplacian   $(-\Delta)^{\delta}$ ($\delta\in\mathbb{R}$) can be defined through the Fourier transform, namely
\begin{equation}\label{1-4c}
(-\Delta)^{\delta}f(x)=\Lambda^{2\delta}f(x)=\int_{\mathbb{R}^3}|x|^{2\delta}\hat{f}(\xi)e^{2\pi ix\cdot\xi}d\xi,
\end{equation}
where  $\widehat{f}$ is the Fourier transform of $f$. Moreover,  $\nabla^l$ with an integral $l\geq0$ stands for the usual spatial derivatives of order $l$. If $l<0$ or $l$ is not a positive integer, $\nabla^l$ stands for $\Lambda^l$ defined by (\ref{1-4c}). We also use $\dot{H}^s(\mathbb{R}^3)$ ($s\in\mathbb{R}$) to denote the homoegneous Sobolev spaces on $\mathbb{R}^3$ with the norm $\|\cdot\|_{H^s}$ defined by $\|f\|_{H^s}:=\|\Lambda^sf\|_{L^2}$, and we use $H^s(\mathbb{R}^3)$ and $L^p(\mathbb{R}^3)$ ($1\leq p\leq\infty$) to describe the usual Sobolev spaces with the norm $\|\cdot\|_{H^s}$ and the usual $L^p$ space with the norm $\|\cdot\|_{L^p}$.
We use the notation $A\lesssim B$ to mean that $A\leq cB$ for a universal constant $c>0$ that only depends on the parameters coming from the problem and the indexes $N$ and $s$  coming from the regularity on the data. We also employ $C$ for positive constant depending additionally on the initial data.
\end{remark}

The first purpose of this paper is to consider the  global well-posedness  of solutions for problem (\ref{1-1}). More precisely, we have the following theorem:
\begin{theorem}
\label{thm1.1}
Let $N\geq2$, assume that $u_0\in H^N(\mathbb{R}^3) $  and there exists a constant $\varepsilon_0>0$ such that if
\begin{equation}
\label{1-2}
\|u_0 \|_{H^2} \leq\varepsilon_0,
\end{equation}
then there exists a unique global solution $u(x,t)$ satisfying that for all $t\geq0$,
\begin{equation}
\label{1-3}\begin{aligned}
\|u(t)\|_{H^N}^2
 +\int_0^t(\|\nabla\Delta u(s)\|_{H^N}^2+\|\Delta  u(s)\|_{H^N}^2 )ds
 \leq C \|u_0\|_{H^N}^2 .\end{aligned}
\end{equation}
\end{theorem}

The asymptotic behavior of solutions is also an interesting topic in the study of the Cauchy problem of dissipative equations. In this paper, we also want to establish the results on the asymptotic behavior, i.e., we show the solutions of problem (\ref{1-1}) satisfy some negative exponent decay rate:
\begin{theorem}
\label{thm1.2}
Under the assumptions of Theorem \ref{thm1.1}, and assuming that  $ u_0 \in\dot{H}^{-s}(\mathbb{R}^3)$ for some $s\in[0,\frac12]$, then for all $t\geq0$,
\begin{equation}
\label{1-4}
\|\Lambda^{-s}u(t)\|_{L^2}  \leq C,
\end{equation}
and
\begin{equation}\label{1-5}
\|\nabla^lu(t)\|_{H^{N-l}} \leq C(1+t)^{-\frac{l+s}4},\quad\hbox{for}~l=0,1,\cdots, N-1 .
\end{equation}\end{theorem}

Note that the Hardy-Littlewood-Sobolev theorem implies that for $p\in[\frac32,2]$,  $L^{p}(\mathbb{R}^3)\subset\dot{H}^{-s}(\mathbb{R}^3)$ with $s=3(\frac1p-\frac12)\in[0,\frac12]$. Therefore, based on Theorem \ref{thm1.1}, we   obtain the following corollary on the optimal decay estimates.
\begin{corollary}
\label{cor1.1}Under the assumptions of Theorem \ref{thm1.1}, if we replace the $\dot{H}^{-s}(\mathbb{R}^3)$ assumption by $ u_0 \in L^{p}(\mathbb{R}^3)$ $(\frac32\leq p\leq2)$, then the following decay estimate holds:
\begin{equation}
\label{1-9}
\|\nabla^lu(t)\|_{H^{N-l}} \leq C(1+t)^{-\sigma_l},\quad\hbox{for}~l=0,1,\cdots, N-1,
\end{equation}where $$\sigma_l=\frac34\left(\frac1p -\frac12\right)+\frac l4 .
$$
\end{corollary}

\begin{remark}
There's an amazing phenomenon: although   the problem we considered is a sixth-order parabolic equation,  the temporary decay of it satisfies (\ref{1-5}) and (\ref{1-9}), which is equivalent to the decay rate of fourth-order generalized heat equation
 \begin{equation} \label{heat}
\left\{\begin{aligned}
      & u_t +\Delta ^2u=0, \quad x\in\mathbb{R}^3,~t\geq0,\\
       &u(x,0)=u_0(x).
                  \end{aligned}\right.
                          \end{equation}
  That is because of the introduce of the fourth order linear term $\kappa_0\Delta^2u$ in (\ref{1-1}). After introduce this term,  the last term on the right hand side, which can be seen as a second order nonlinear term, can be controlled by it. Since both $-\Delta^3u$ and $ \Delta^2u$ are ``good" term, we can use the lower order one to study the decay estimate. Hence, the  decay of problem (\ref{1-1})   is equivalent to the decay rate of fourth order generalized heat equation.
\end{remark}



The main difficulties to consider the Cauchy problem of   sixth-order   Cahn-Hilliard equation  arising in oil-water-surfactant mixtures in $\mathbb{R}^3$ are how  to deal with the second order nonlinear term $\Delta f(u)=\Delta(u-1)^2(u+1)^2(u^2+h_0)$ and how to obtain a negative Sobolev estimates to study the decay rate of solutions. Since the principle part of  problem (\ref{1-1}) is a six-order linear term and the nonlinear term $\Delta f(u)$ is only   second-order. Due to Sobolev's embedding theorem in $\mathbb{R}^3$, we can't control a second-order nonlinear term through sixth-order linear term. In order to overcome this difficulty, we borrow a fourth-order term from the other nonlinear term $\Delta(a(u)\Delta u)$, rewrite (\ref{1-1f})$_1$ as (\ref{1-1}).
Hence, one can control the second-order nonlinear term by the last term of the left hand side of equation  (\ref{1-1}). On the other hand, to consider the temporal decay rate of solutions of dissipative equations, one of the main tools is the standard Fourier splitting method \cite{S1,S2}. By using this method, lots of decay problems were solved (see e.g., \cite{Anh,H4,NS,Dai,B} and the reference therein). In this paper, since there exists the lower order linear term in the right hand side of problem (\ref{1-1}), it is difficulty for us to use Fourier splitting method to study the decay rate of solutions.  By using the pure energy method \cite{GW,W,T}  of using a family of scaled energy estimates with minimum derivative counts and interpolations among them, we overcome the difficulty caused by the lower order linear term of right hand side of (\ref{1-1}), obtained the suitable a priori estimates in the Sobolev space $H^N$ and the negative Sobolev space $\dot{H}^{-s}$ ($0\leq s\leq\frac12$),   establish the   optimal decay rate of problem (\ref{1-1}) in $\mathbb{R}^3$.

The structure of this paper is organized as follows.  In Section 2, we introduce some preliminary results, which are useful to prove our main results. Section 3 is  devoted to prove the small data global well-posedness of problem (\ref{1-1}). In the last section,  we establish the time decay rate of solutions.

\section{Preliminaries}
First of all, we introduce  the Kato-Ponce inequality which is of great importance in this paper.
\begin{lemma}[\cite{KP}]\label{Kato}
Let $1<p<\infty$, $s>0$. There exists a positive constant $C$ such that
\begin{equation}
\label{K-1}
\|\Lambda^s(fg)-f\Lambda^sg\|_{L^p}\leq C(\|\nabla f\|_{L^{p_1}}\|\Lambda^{s-1}g\|_{L^{p_2}}+\|\Lambda^sf\|_{L^{q_1}}\|g\|_{L^{q_2}}),
\end{equation}
and
\begin{equation}
\label{K-2}
\|\Lambda^s{fg}\|_{L^p}\leq C(\|f\|_{L^{p_1}}\|\Lambda^sg\|_{L^{p_2}}+\|\Lambda^sf\|_{L^{q_1}}\|g\|_{L^{q_2}},
\end{equation}
where $p_2,q_2\in(1,\infty)$ satisfying $\frac1p=\frac1{p_1}+\frac1{p_2}=\frac1{q_1}+\frac1{q_2}$.
\end{lemma}

The following Gagliardo-Nirenberg inequality was proved in \cite{Nirenberg}.
\begin{lemma}[\cite{Nirenberg}]
\label{lem2.1}
Let $0\leq m,\alpha\leq l$, then we have
\begin{equation}\label{2-1}
\|\nabla^{\alpha}f\|_{L^p(\mathbb{R}^3)}\lesssim\|\nabla^mf\|_{L^q(\mathbb{R}^3)}^{1-\theta}\|\nabla^lf\|_{L^r(\mathbb{R}^3)}^{\theta},
\end{equation}
where $\theta\in[0,1]$ and $\alpha$ satisfies
\begin{equation}\label{2-2}
\frac{\alpha}3-\frac1p=\left(\frac m3-\frac1q\right)(1-\theta)+\left(\frac l3-\frac1r\right)\theta.
\end{equation}
Here, when $p=\infty$, we require that $0<\theta<1$.
\end{lemma}

There's a Sobolev embedding for the homogeneous space $\dot{H}^s$.
\begin{lemma}[\cite{Agm}]
\label{S}
There exists a constant $c$ such that for $0\leq s<\frac32$,
\begin{equation}\label{s1}
\|u\|_{L^{\frac6{3-2s}}}\leq c\|u\|_{\dot{H}^s}\quad\hbox{for ~all}~~u\in\dot{H}^s(\mathbb{R}^3).
\end{equation}
\end{lemma}
We now give  Agmon's inequality   in the following lemma.
\begin{lemma}[\cite{Agm}]
\label{Agm}
If $u\in H^2(\mathbb{R}^3)$, then $u$ is (almost everywhere equal to) a continuous function and
\begin{equation}\label{agm}
\|u\|_{L^{\infty}}\lesssim \|\nabla u\|_{L^2}^{\frac12}\|\Delta u\|_{L^2}^{\frac12}.
\end{equation}
\end{lemma}
We also introduce the Hardy-Littlewood-Sobolev theorem, which implies the following $L^p$ type inequality.
\begin{lemma}[\cite{Stein,15}]
\label{lem2.3}
Let $0\leq s<\frac32$, $1<p\leq 2$ and $\frac12+\frac s3=\frac1p$, then
\begin{equation}
\label{2-4}
\|f\|_{\dot{H}^{-s}(\mathbb{R}^3)}\lesssim\|f\|_{L^p(\mathbb{R}^3)}.
\end{equation}
\end{lemma}

The following special Sobolev interpolation lemma will be used in the proof of Theorem \ref{thm1.1}.
\begin{lemma}[\cite{W,GW}]
\label{lem2.2}
Let $s,k\geq0$ and $l\geq0$, then
\begin{equation}
\label{2-3}
\|\nabla^lf\|_{L^2(\mathbb{R}^3)}\leq\|\nabla^{l+k}f\|_{L^2(\mathbb{R}^3)}^{ 1-\theta }\|f\|_{\dot{H}^{-s}(\mathbb{R}^3)}^{ \theta },\quad\hbox{with}~\theta=\frac2{l+k+s}.
\end{equation}
\end{lemma}

\section{Proof of Theorem \ref{thm1.1}}

\subsection{Energy estimates}

The purpose of this subsection is to establish the a priori nonlinear energy estimates for problem (\ref{1-1}). Hence, we suppose that for sufficiently small $\varepsilon>0$,
\begin{equation}
\label{3-1}
\sqrt{\mathcal{E}_0^2 (t)}=\|u(t)\|_{H^2} \leq\varepsilon.
\end{equation}

We begin with the energy estimates including $u$ itself.
\begin{lemma}
\label{lem3.1}
Suppose that all assumptions in Theorem \ref{thm1.1} hold.
If $\sqrt{\mathcal{E}_0^2(t)}\leq\delta$, then for $k=0,1,\cdots, N$, we have
\begin{equation}\begin{aligned}
\label{3-2}&\frac d{dt}\int_{\mathbb{R}^3} |\nabla^ku|^2dx
 + (\|\nabla^{k+3}u\|_{L^2}^2+\|\nabla^{k+2}u\|_{L^2}^2 )\lesssim(\varepsilon^2+\varepsilon^5)(\|\nabla^{k+3}u\|_{L^2}^2+\|\nabla^{k+2}u\|_{L^2}^2 ).
\end{aligned}\end{equation}
\end{lemma}
\begin{proof}
Applying $\nabla^k$ to (\ref{1-1})$_1$,   multiplying the resulting identity by $\nabla^ku$, and then integrating over $\mathbb{R}^3$ by parts, we arrive at
\begin{equation}
\begin{aligned}\label{3-3}&
\frac12\frac d{dt}\|\nabla^ku\|_{L^2}^2+\delta\|\nabla^{k+3}u\|_{L^2}^2+\kappa_0\|\nabla^{k+2}u\|_{L^2}^2
\\
=&-\kappa_1\int_{\mathbb{R}^3}\nabla^k\left\{\Delta\left[\left(u+ \kappa_2\right)\left(u-\kappa_2\right)\Delta u\right]\right\}\cdot\nabla^kudx
 -\kappa_1\int_{\mathbb{R}^3}\nabla^k\left[\Delta\left(u|\nabla u|^2\right)\right]\cdot\nabla^kudx
\\&+\int_{\mathbb{R}^3}\nabla^k\Delta\left[(u-1)^2(u+1)^2(u^2+h_0)\right]\cdot\nabla^kudx
\\
=&I_1+I_2+I_3
.
\end{aligned}
\end{equation}
We can estimate the first term of the right hand side of (\ref{3-3}) as
\begin{equation}
\begin{aligned}
\label{3-4}
I_1=&-\kappa_1\int_{\mathbb{R}^3}\nabla^k\left\{\Delta\left[\left(u+ \kappa_2\right)\left(u-\kappa_2\right)\Delta u\right]\right\}\cdot\nabla^kudx
\\
=&-\kappa_1\int_{\mathbb{R}^3}\nabla^k\left[\left(u+ \kappa_2\right)\left(u-\kappa_2\right)\Delta u\right] \cdot\nabla^k\Delta u dx
\\
=&-\kappa_1\sum_{0\leq l\leq k}\sum_{0\leq m\leq l}C_k^lC_l^m\int_{\mathbb{R}^3}\nabla^m(u+\kappa_2)\cdot\nabla^{l-m}(u-\kappa_2)\cdot\nabla^{k-l}\Delta u\cdot\nabla^k\Delta udx
\\
\lesssim& \|\nabla^m(u+\kappa_2)\|_{L^6}\|\nabla^{l-m}(u-\kappa_2)\|_{L^6}\|\nabla^{k-l+2} u\|_{L^6}\|\nabla^{k+2} u\|_{L^2}
\\
\lesssim&\|\nabla^{m+1} u\|_{L^2}\|\nabla^{l-m+1}u\|_{L^2}\|\nabla^{k-l+3} u\|_{L^2}\|\nabla^{k+2} u\|_{L^2}
\\
\lesssim& \|\nabla^{k+3}u\|_{L^2}^{\frac m{k+2}}\|\nabla u\|_{L^2}^{1-\frac m{k+2}}\|\nabla^{k+3}u\|_{L^2}^{\frac{l-m}{k+2}}\|\nabla u\|_{L^2}^{1-\frac{l-m}{k+2}}\|\nabla^{k+3}u\|_{L^2}^{1-\frac l{k+2}}\|\nabla u\|_{L^2}^{\frac l{k+2}}\|\nabla^{k+2} u\|_{L^2}
\\
\lesssim& \|\nabla u\|_{L^2}^2(\|\nabla^{k+3}u\|_{L^2}^2+\|\nabla^{k+2}u\|_{L^2}^2)
\\
\lesssim& \varepsilon^2(\|\nabla^{k+3}u\|_{L^2}^2+\|\nabla^{k+2}u\|_{L^2}^2).
\end{aligned}
\end{equation}
For the second  term of the right hand side of (\ref{3-3}), we have
\begin{equation}
\begin{aligned}
\label{3-5}
I_2=&-\kappa_1\int_{\mathbb{R}^3}\nabla^k\left[\Delta\left(u|\nabla u|^2\right)\right]\cdot\nabla^kudx
\\
=&-\kappa_1\int_{\mathbb{R}^3}\nabla^k \left(u|\nabla u|^2\right) \cdot\nabla^k\Delta u dx
\\
=&-\kappa_1\sum_{0\leq l\leq k}\sum_{0\leq m\leq l}C_k^lC_l^m\int_{\mathbb{R}^3}\nabla^mu\cdot\nabla^{l-m}\nabla u\cdot\nabla^{k-l}\nabla u\cdot\nabla^k\Delta u dx
\\
\lesssim&\|\nabla^mu\|_{L^6}\|\nabla^{l+1-m}u\|_{L^6}\|\nabla^{k+1-l}u\|_{L^6}\|\nabla^{k+2}u\|_{L^2}
\\
\lesssim&\|\nabla^{m+1}u\|_{L^2}\|\nabla^{l+2-m}u\|_{L^2}\|\nabla^{k+2-l}u\|_{L^2}\|\nabla^{k+2}u\|_{L^2}
\\
\lesssim&\|\nabla^{k+3}u\|_{L^2}^{\frac m{k+2}}\|\nabla u\|_{L^2}^{1-\frac m{k+2}}\|\nabla^{k+3}u\|_{L^2}^{\frac{l+1-m}{k+2}}\|\nabla u\|_{L^2}^{1-\frac{l+1-m}{k+2}}
\\
&\cdot\|\nabla^{k+3}u\|_{L^2}^{\frac{k+1-l}{k+2}}\|\nabla u\|_{L^2}^{1-\frac{k+1-l}{k+2}}\|\nabla^{k+2}u\|_{L^2}
\\
\lesssim&\|\nabla u\|_{L^2}^2(\|\nabla^{k+3}u\|_{L^2}^2+\|\nabla^{k+2}u\|_{L^2}^2)\\
\lesssim& \varepsilon^2(\|\nabla^{k+3}u\|_{L^2}^2+\|\nabla^{k+2}u\|_{L^2}^2).
\end{aligned}
\end{equation}
Moreover, the last term of the right hand side of (\ref{3-3}) satisfies
\begin{equation}
\begin{aligned}\label{3-6}
I_3=&\int_{\mathbb{R}^3}\nabla^k\Delta\left[(u-1)^2(u+1)^2(u^2+h_0)\right]\cdot\nabla^kudx
\\
=&\int_{\mathbb{R}^3}\nabla^k\left[(u-1)^2(u+1)^2(u^2+h_0)\right]\cdot\nabla^k\Delta udx
\\
=&\int_{\mathbb{R}^3}\nabla^k\left[(u-1)^2(u+1)^2 u^2 \right]\cdot\nabla^k\Delta udx
\\&+h_0\int_{\mathbb{R}^3}\nabla^k\left[(u-1)^2(u+1)^2 \right]\cdot\nabla^k\Delta udx
\\
=&I_{31}+I_{32}
.
\end{aligned}
\end{equation}
Note that
\begin{equation}
\begin{aligned}\label{3-7}
I_{31}=&\int_{\mathbb{R}^3}\nabla^k\left[(u-1)^2(u+1)^2 u^2 \right]\cdot\nabla^k\Delta udx
\\
=&\sum_{0\leq k\leq k}\sum_{0\leq m\leq l}C_k^lC_l^m\int_{\mathbb{R}^3}\nabla^m(u-1)^2\cdot\nabla^{l-m}(u+1)^2\cdot\nabla^{k-l}u^2\cdot\nabla^{k}\Delta udx
\\
\lesssim&\|\nabla^m(u-1)^2\|_{L^6}\|\nabla^{l-m}(u+1)^2\|_{L^6}\|\nabla^{k-l}u^2\|_{L^6}\|\nabla^{k}\Delta u\|_{L^2}
\\
\lesssim&\|u-1\|_{L^{\infty}}\|\nabla^m(u-1)\|_{L^6}\|u+1\|_{L^{\infty}}\|\nabla^{l-m}(u+1)\|_{L^6}\|u\|_{L^{\infty}}\|\nabla^{k-l}u\|_{L^6}\|\nabla^{k+2}u\|_{L^2}
\\
\lesssim&(\|\nabla u\|_{L^2}^{\frac12}\|\nabla^2u\|_{L^2}^{\frac12})^3\|\nabla^{m+1}u\|_{L^2}\|\nabla^{l-m+1}u\|_{L^2}\|\nabla^{k-l+1}u\|_{L^2}\|\nabla^{k+2}u\|_{L^2}
\\
\lesssim&(\|\nabla u\|_{L^2}^3+\|\nabla^2u\|_{L^2}^3)
\\
&\cdot(\|\nabla^{k+2}u\|_{L^2}^{\frac{m+1}{k+2}}\|u\|_{L^2}^{1-\frac{m+1}{k+2}}\|\nabla^{k+2}u\|_{L^2}^{\frac{l-m+1}{k+2}}\|u\|_{L^2}^{1-\frac{l-m+1}{k+2}}
\|\nabla^{k+3}u\|_{L^2}^{\frac{k-l}{k+2}}\|\nabla u\|_{L^2}^{1-\frac{k-l}{k+2}}\|\nabla^{k+2}u\|_{L^2}
\\
\lesssim&\|u\|_{H^2}^5( \|\nabla^{k+3}u\|_{L^2}^2+\|\nabla^{k+2}u\|_{L^2}^2)\\
\lesssim&\varepsilon^5( \|\nabla^{k+3}u\|_{L^2}^2+\|\nabla^{k+2}u\|_{L^2}^2),
\end{aligned}
\end{equation}
 and
 \begin{equation}
\begin{aligned}\label{3-8}
I_{32}=&h_0\int_{\mathbb{R}^3}\nabla^k\left[(u-1)^2(u+1)^2 \right]\cdot\nabla^k\Delta udx
\\
\\
=&h_0\sum_{0\leq k\leq k}\sum_{0\leq m\leq l}C_k^lC_l^m\int_{\mathbb{R}^3}\nabla^m(u-1) \cdot\nabla^{l-m}(u-1)\cdot\nabla^{k-l}(u+1)^2\cdot\nabla^{k}\Delta udx
\\
\lesssim&\|\nabla^m(u-1) \|_{L^6}\|\nabla^{l-m}(u-1)\|_{L^6}\|\nabla^{k-l}(u+1)^2\|_{L^6}\|\nabla^{k}\Delta u\|_{L^2}
\\
\lesssim&\|u+1\|_{L^{\infty}}\|\nabla^m(u-1)\|_{L^6} \|\nabla^{l-m}(u-1)\|_{L^6} \|\nabla^{k-l}(u+1)\|_{L^6}\|\nabla^{k+2}u\|_{L^2}
\\
\lesssim& \|\nabla u\|_{L^2}^{\frac12}\|\nabla^2u\|_{L^2}^{\frac12} \|\nabla^{m+1}u\|_{L^2}\|\nabla^{l-m+1}u\|_{L^2}\|\nabla^{k-l+1}u\|_{L^2}\|\nabla^{k+2}u\|_{L^2}
 \\
\lesssim&\|u\|_{H^2}^3( \|\nabla^{k+3}u\|_{L^2}^2+\|\nabla^{k+2}u\|_{L^2}^2)\\
\lesssim&\varepsilon^3( \|\nabla^{k+3}u\|_{L^2}^2+\|\nabla^{k+2}u\|_{L^2}^2).
\end{aligned}
\end{equation}
Plugging the estimates (\ref{3-4})-(\ref{3-8}) into (\ref{3-3}), because $\varepsilon$ is small, we then obtain (\ref{3-2}) and complete the proof.
\end{proof}
\subsection{Local well-posedness}

In this subsection, we prove the local well-posedness of solution $ u(t) $ in $H^2$-norm.

We first construct the solution $ (u^j )_{j\geq0}$ by solving iteratively the Cauchy problem:
\begin{equation} \label{4-1}
\left\{ \begin{aligned}
&\partial_tu^{j+1}-\delta\Delta^3u^{j+1}+\kappa_0\Delta^2u^{j+1}=-\Delta\left[ \kappa_1\left(u^j+ \kappa_2\right)\left(u^j-\kappa_2\right)\Delta u^{j+1}\right.
\\
&\quad\quad\quad\quad\quad\quad\quad\quad\!\!\!\left.+\kappa_1u^j \nabla u^j\cdot\nabla u^{j+1}-(u^j-1)^2(u^j+1) [(u^j)^2+h_0](u^{j+1}+1) \right] ,\\
                 & u^{j+1} |_{t=0}= u_0(x) ,\quad x\in \mathbb{R}^3,
                          \end{aligned} \right.
                          \end{equation}
for $j\geq0$, where $ u^0 \equiv0$ holds. One denote $( u^j )_{j\geq0}$ in short hand by $(\mathcal{A}_j)_{j\geq0}$  and denote $ u_0$ as $\mathcal{A}_0$.
In the following, we shall show that $(\mathcal{A}^j)_{j\geq0}$ is a Cauchy sequence in Banach space $C([0,T_1];H^2)$ with $T_1>0$ suitable small. Then, by take limit and continuous argument, one propose to prove that $ u(t)$ is a global solution to Cauchy problem (\ref{1-1}).

\begin{lemma}
\label{lem4.1}
Suppose that all assumptions in Theorem \ref{thm1.1} hold. There are constants $\varepsilon_1>0$, $T_1>0$ and $M_1>0$ such that if $\|\mathcal{A}_0\|_{H^2}\leq\varepsilon_1$, then for each $j\geq0$, $\mathcal{A}^j\in C([0,T_1];H^2)$ is well defined and
\begin{equation}
\label{4-2}
\sup_{0\leq t\leq T_1}\|\mathcal{A}^j(t)\|_{H^2}\leq M_1,\quad j\geq0.
\end{equation}
Moreover, $(\mathcal{A}^j)_{j\geq0}$ is a Cauchy sequence in Banach space $C([0,T_1];H^2)$, the corresponding limit function denoted by $\mathcal{A}(t)$ belongs to $C([0,T_1];H^2)$ with
\begin{equation}
\label{4-3}
\sup_{0\leq t\leq T_1}\|\mathcal{A}(t)\|_{H^2}\leq M_1,
\end{equation}
and $\mathcal{A}=u(t)$ is a solution over $[0,T_1]$ to problem (\ref{1-1}). Finally, for the Cauchy problem (\ref{1-1}), there exists at most one solution  $ u(t)$ in $C([0,T_1];H^2)$ satisfying (\ref{4-3}).
\end{lemma}
\begin{proof}The inequality(\ref{4-2}) will be proved by induction. By using the assumption at initial step, we get $\mathcal{A}_0=0$, which means $j=0$ holds. Next, suppose that (\ref{4-2}) holds for $j\geq0$ with $M_1>0$ small enough to be determined later, we are going to prove it also holds for $j+1$. Hence, we need some energy estimates on $\mathcal{A}^{j+1}$.  On the basis of (\ref{4-1})$_1$, we obtain for $k=2$ and $0\leq l\leq k$,
\begin{equation}
\begin{aligned}
\label{4-4}
&\frac12\frac d{dt} \|\nabla^lu^{j+1}\|_{L^2}^2 +\delta\|\nabla^{l+3}u^{j+1}\|_{L^2}^2+\kappa\|\nabla^{l+2}u^{j+1}\|_{L^2}^2
\\
=&-\kappa_1\int_{\mathbb{R}^3}\nabla^l[(u^j+\kappa_2)(u^j-\kappa_2)\Delta u^{j+1}]\cdot\nabla^l\Delta u^{j+1}dx
\\
&-\kappa_1\int_{\mathbb{R}^3}\nabla^l[u^j\nabla u^j\cdot\nabla u^{j+1}]\cdot\nabla ^l\Delta u^jdx
\\
&-2\int_{\mathbb{R}^3}\nabla^{l} [(u^j-1) (u^j+1)^2((u^j)^2+h_0)\nabla u^{j+1}]\cdot\nabla^{l+1}u^{j+1}dx
\\&-2\int_{\mathbb{R}^3}\nabla^{l} [(u^j+1) (u^j-1)^2((u^j)^2+h_0)\nabla u^{j+1}\cdot\nabla^{l+1}u^{j+1}dx\\&-2\int_{\mathbb{R}^3}\nabla^{l} [(u^j-1) ^2 (u^j+1)^2 \nabla u^{j+1}]\cdot\nabla^{l+1}u^{j+1}dx
\\
=&J_1+J_2+J_3+J_4+J_5.
\end{aligned}
\end{equation}
Since it is trivial for the case $l=0$, thus, we   only need to consider $l=1, 2$. For the term $J_1$, if $l=1$, we estimate as follows
\begin{equation}
\begin{aligned}\label{4-5}
J_1=&-\kappa_1\int_{\mathbb{R}^3}\nabla [(u^j+\kappa_2)(u^j-\kappa_2)\Delta u^{j+1}]\cdot\nabla \Delta u^{j+1}dx
\\
\lesssim&\|\nabla^3 u^{j+1}\|_{L^2}\|\nabla[(u^j+\kappa_2)(u^j-\kappa_2)\Delta u^{j+1}]\|_{L^{2}}
\\
\lesssim&\sum_{s=0}^1\sum_{m=0}^sC_1^sC_s^m\|\nabla^3u^{j+1}\|_{L^2}\|\nabla^m(u^j+\kappa_2)\|_{L^6}\|\nabla^{s-m}(u^j-\kappa_2)\|_{L^6}\|\nabla^{1-s+2}u^{j+1}\|_{L^6}
\\
\lesssim&\sum_{s=0}^1\sum_{m=0}^sC_1^sC_s^m\|\nabla^3u^{j+1}\|_{L^2}\|\nabla^{m+1}u^j \|_{L^2}\|\nabla^{s-m+1}u^j \|_{L^2}\|\nabla^{1-s+3}u^{j+1}\|_{L^2}
\\
\lesssim&\|u^j\|_{H^2}^2(\|\nabla^2u^{j+1}\|_{H^2}^2+\|\nabla^3u^{j+1}\|_{H^2}^2) .
\end{aligned}
\end{equation}
For the case $l=2$, we can estimate $J_1$ as
\begin{equation}
\begin{aligned}\label{4-5-1}
J_1
\lesssim&\|\nabla^4u^{j+1}\|_{L^6}\|\nabla^2 [(u^j+\kappa_2)(u^j-\kappa_2)\Delta u^{j+1}]\|_{L^{\frac65}}
\\
\lesssim&\|\nabla^4u^{j+1}\|_{L^6}\left\|\sum_{s=0}^2\sum_{m=0}^s\nabla^m(u^j+\kappa_2)\cdot\nabla^{s-m}(u^j-\kappa_2)\cdot\nabla^{4-s}u^{j+1}\right\|_{L^{\frac65}}
\\
\lesssim&\|\nabla^5u^{j+1}\|_{L^2}\left(\underbrace{\|u^j+\kappa_2)\|_{L^6}\|u^j-\kappa_2\|_{L^6}\|\nabla^4u^{j+1}\|_{L^2}}_{s=m=0}\right.
\\
&
+\underbrace{\|u^j+\kappa_2\|_{L^6}\|\nabla u^j \|_{L^6}\|\nabla^3u^{j+1}\|_{L^2}}_{s=1,m=0}
 +\underbrace{\|\nabla u^j \|_{L^6}\| u^j-\kappa_2 \|_{L^6}\|\nabla^3u^{j+1}\|_{L^2}}_{s=m=1}\\
&+\underbrace{\| u^j+\kappa_2 \|_{L^6}\|\nabla^2  u^j  \|_{L^2}\|\nabla^2u^{j+1}\|_{L^6}}_{s=2,m=0}
 +\underbrace{\|\nabla  u^j  \|_{L^6}\|\nabla  u^j  \|_{L^2}\|\nabla^2u^{j+1}\|_{L^6}}_{s=2,m=1}\\
&\left.+\underbrace{\|\nabla^2( u^j+\kappa_2) \|_{L^2}\| u^j-\kappa_2   \|_{L^6}\|\nabla^2u^{j+1}\|_{L^6}}_{s=2,m=2}\right)
\\
\lesssim&\|u^j\|_{H^2}^2(\|\nabla^2u^{j+1}\|_{H^2}^2+\|\nabla^3u^{j+1}\|_{H^2}^2) .
\end{aligned}
\end{equation}
For the term $J_2$, if $l=1$, we have
\begin{equation}
\begin{aligned}\label{4-6}
J_2=&-\kappa_1\int_{\mathbb{R}^3}\nabla [u^j\nabla u^j\cdot\nabla u^{j+1}]\cdot\nabla  \Delta u^jdx
\\
\lesssim&\|\nabla^3 u^{j+1}\|_{L^6}\|\nabla [u^j\nabla u^j\cdot\nabla u^{j+1}]\|_{L^{\frac65}}
\\
\lesssim&\sum_{0\leq s\leq 1 }\sum_{0\leq m\leq s}\|\nabla^3 u^{j+1}\|_{L^6}\|\nabla^mu^j\cdot\nabla^{s-m}\nabla u^j\cdot\nabla^{1-s}\nabla u^{j+1} \|_{L^{\frac65}}
\\
\lesssim&\sum_{0\leq s\leq 1 }\sum_{0\leq m\leq s}\|\nabla^3 u^{j+1}\|_{L^6}\|\nabla^mu^j\|_{L^6}\| \nabla^{s+1-m}  u^j\|_{L^2}\|\nabla^{1-s+1} u^{j+1} \|_{L^{6}}\\
\lesssim&\sum_{0\leq s\leq 1 }\sum_{0\leq m\leq s}\|\nabla^4 u^{j+1}\|_{L^2}\|\nabla^{m+1}u^j\|_{L^2}\| \nabla^{s+1-m}  u^j\|_{L^2}\|\nabla^{1-s+2} u^{j+1} \|_{L^{2}}
\\
\lesssim&\|u^j\|_{H^2}^2(\|\nabla^2u^{j+1}\|_{H^2}^2+\|\nabla^3u^{j+1}\|_{H^2}^2) .
\end{aligned}
\end{equation}
For the case $l=2$ of $J_2$, we estimate as
\begin{equation}
\begin{aligned}
\label{4-6-1}
J_2\lesssim&\sum_{s=0}^1\sum_{m=0}^s \|\nabla^4 u^{j+1}\|_{L^6}\|\nabla^mu^j\cdot\nabla^{s-m}\nabla u^j\cdot\nabla^{2-s}\nabla   u^{j+1}]\|_{L^{\frac65}}
\\
&\underbrace{-\kappa_1 \sum_{0\leq m\leq 2}C_2^m\int_{\mathbb{R}^3}\nabla^m u^j\cdot\nabla^{2-m}\nabla u^j\cdot\nabla u^{j+1} \cdot\nabla^2  \Delta u^{j+1}dx}_{s=2}
\\
=&J_{21}+J_{22}.
\end{aligned}
\end{equation}
Note that
\begin{equation}
\begin{aligned}\label{4-6-2}
J_{21}=&\sum_{s=0}^1\sum_{m=0}^s  \|\nabla^4 u^{j+1}\|_{L^6}\|\nabla^mu^j\cdot\nabla^{s-m}\nabla u^j\cdot\nabla^{2-s}\nabla   u^{j+1}]\|_{L^{\frac65}}
\\
\lesssim&\|\nabla^5u^{j+1}\|_{L^2}\left(\underbrace{\|u^j\|_{L^6}\|\nabla u^j\|_{L^6}\|\nabla^3 u^{j+1}\|_{L^2}}_{s=m=0}+\underbrace{\|u^j\|_{L^6}\|\nabla^2 u^j\|_{L^2}\|\nabla^2 u^{j+1}\|_{L^6}}_{s=1,m=0}\right.
\\&\left.+\underbrace{\|\nabla u^j\|_{L^6}\|\nabla u^j\|_{L^2}\|\nabla^2u^{j+1}\|_{L^6}}_{s=m=1}\right)
\\
\lesssim&\|u^j\|_{H^2}^2(\|\nabla^2u^{j+1}\|_{H^2}^2+\|\nabla^3u^{j+1}\|_{H^2}^2),
\end{aligned}
\end{equation}
and
\begin{equation}
\begin{aligned}
\label{4-6-3}
J_{22}=&-\kappa_1 \sum_{0\leq m\leq 2}C_2^m\int_{\mathbb{R}^3}\nabla^m u^j\cdot\nabla^{2-m}\nabla u^j\cdot\nabla u^{j+1} \cdot\nabla^2  \Delta u^{j+1}dx
\\
=& -\kappa_1  \int_{\mathbb{R}^3}  u^j\cdot\nabla^{2 }\nabla u^j\cdot\nabla u^{j+1} \cdot\nabla^2  \Delta u^{j+1}dx
-2\kappa_1  \int_{\mathbb{R}^3}  \nabla u^j\cdot\nabla \nabla u^j\cdot\nabla u^{j+1} \cdot\nabla^2  \Delta u^{j+1}dx
\\&-\kappa_1  \int_{\mathbb{R}^3}  \nabla^2 u^j\cdot \nabla u^j\cdot\nabla u^{j+1} \cdot\nabla^2  \Delta u^{j+1}dx
\\
=&\kappa_1\int_{\mathbb{R}^3} \nabla  u^j\cdot\nabla \nabla u^j\cdot\nabla u^{j+1} \cdot\nabla^2  \Delta u^{j+1}dx+
\kappa_1  \int_{\mathbb{R}^3}  u^j\cdot\nabla \nabla u^j\cdot\nabla^2 u^{j+1} \cdot\nabla^2  \Delta u^{j+1}dx\\&
+\kappa_1  \int_{\mathbb{R}^3}  u^j\cdot\nabla \nabla u^j\cdot\nabla u^{j+1} \cdot\nabla^3  \Delta u^{j+1}dx
-2\kappa_1  \int_{\mathbb{R}^3}  \nabla u^j\cdot\nabla \nabla u^j\cdot\nabla u^{j+1} \cdot\nabla^2  \Delta u^{j+1}dx
\\&-\kappa_1  \int_{\mathbb{R}^3}  \nabla^2 u^j\cdot \nabla u^j\cdot\nabla u^{j+1} \cdot\nabla^2  \Delta u^{j+1}dx
\\
\lesssim&\|\nabla u^j\|_{L^6}\|\nabla^2u^j\|_{L^2}\|\nabla u^{j+1}\|_{L^6}\|\nabla^4u^{j+1}\|_{L^6}
+\|  u^j\|_{L^6}\|\nabla^2u^j\|_{L^2}\|\nabla^2 u^{j+1}\|_{L^6}\|\nabla^4u^{j+1}\|_{L^6}
\\&+\|  u^j\|_{L^{\infty}}\|\nabla^2u^j\|_{L^2}\|\nabla u^{j+1}\|_{L^{\infty}}\|\nabla^5u^{j+1}\|_{L^2}
+\|\nabla u^j\|_{L^6}\|\nabla^2u^j\|_{L^2}\|\nabla u^{j+1}\|_{L^6}\|\nabla^4u^{j+1}\|_{L^6}
\\&+\|\nabla^2 u^j\|_{L^2}\|\nabla u^j\|_{L^6}\|\nabla u^{j+1}\|_{L^6}\|\nabla^4u^{j+1}\|_{L^6}
\\
\lesssim&\|u^j\|_{H^2}^2(\|\nabla^2u^{j+1}\|_{H^2}^2+\|\nabla^3u^{j+1}\|_{H^2}^2).
\end{aligned}
\end{equation}
Combining (\ref{4-6-1})-(\ref{4-6-3}) together, if $l=2$, we obtain
\begin{equation}
\label{4-6-4}\begin{aligned}
J_2\lesssim&\|u^j\|_{H^2}^2(\|\nabla^2u^{j+1}\|_{H^2}^2+\|\nabla^3u^{j+1}\|_{H^2}^2).
\end{aligned}
\end{equation}
Next, for $J_3$, if $l=1$, we have
\begin{equation}
\begin{aligned}
\label{4-6-5}
J_3=&-2\int_{\mathbb{R}^3}\nabla  [(u^j-1) (u^j+1)^2((u^j)^2+h_0)\nabla u^{j+1}]\cdot\nabla^{2}u^{j+1}dx
\\
=&-2\int_{\mathbb{R}^3} (u^j+1)^2((u^j)^2+h_0)\nabla   (u^j-1)\cdot\nabla u^{j+1} \cdot\nabla^{2}u^{j+1}dx
\\&
-4\int_{\mathbb{R}^3} (u^j-1)(u^j+1)((u^j)^2+h_0)\nabla   (u^j+1)\cdot\nabla u^{j+1} \cdot\nabla^{2}u^{j+1}dx
\\
&-4\int_{\mathbb{R}^3}(u^j-1) (u^j+1)^2u^j \nabla    u^j \cdot\nabla u^{j+1} \cdot\nabla^{2}u^{j+1}dx
\\
&-2\int_{\mathbb{R}^3} (u^j-1) (u^j+1)^2((u^j)^2+h_0)\nabla^2 u^{j+1} \cdot\nabla^{2}u^{j+1}dx
\\
\lesssim&\|u^j+1\|_{L^6}^2\|(u^j)^2+h_0\|_{L^6}\|\nabla(u^j-1)\|_{L^6}\|\nabla u^{j+1}\|_{L^6}\|\nabla^2u^{j+1}\|_{L^6}
\\
&+\|u^j-1\|_{L^6} \|u^j+1\|_{L^6}\|(u^j)^2+h_0\|_{L^6}\|\nabla(u^j+1)\|_{L^6}\|\nabla u^{j+1}\|_{L^6}\|\nabla^2u^{j+1}\|_{L^6}
\\
&+\|u^j-1\|_{L^6} \|u^j+1\|_{L^6}^2
  \|u^j\|_{L^{\infty}}\|\nabla u^j \|_{L^6}\|\nabla u^{j+1}\|_{L^6}\|\nabla^2u^{j+1}\|_{L^6}
\\
&+ \|u^j-1\|_{L^6} \|u^j+1\|_{L^6}^2\|(u^j)^2+h_0\|_{L^6} \|\nabla^2 u^{j+1}\|_{L^6}^2
\\
\lesssim&\|u^j\|_{H^2}^5(\|\nabla^2u^{j+1}\|_{H^2}^2+\|\nabla^3u^{j+1}\|_{H^2}^2).
\end{aligned}
\end{equation}
Similarly, if $l=1$, we can also obtain
\begin{equation}
\label{4-6-6}
J_4+J_5\lesssim \|u^j\|_{H^2}^5(\|\nabla^2u^{j+1}\|_{H^2}^2+\|\nabla^3u^{j+1}\|_{H^2}^2).
\end{equation}
By Lemma \ref{Kato}, we consider the case $l=2$ of $J_3$:
\begin{equation}
\begin{aligned}
\label{4-6-7}
J_3=&-2\int_{\mathbb{R}^3}\nabla^2  [(u^j-1) (u^j+1)^2((u^j)^2+h_0)\nabla u^{j+1}]\cdot\nabla^{3}u^{j+1}dx
\\
\lesssim&\|\nabla^3u^{j+1}\|_{L^3}\|\nabla^2  [(u^j-1) (u^j+1)^2((u^j)^2+h_0)\nabla u^{j+1}]\|_{L^{\frac32}}
\\
\lesssim&\|\nabla^3u^{j+1}\|_{L^3} \|(u^j-1) (u^j+1)^2((u^j)^2+h_0)\|_{L^6}\|\nabla^3 u^{j+1}\|_{L^2}
\\
&+\|\nabla^3u^{j+1}\|_{L^3}\|\nabla^2 [(u^j-1) (u^j+1)^2((u^j)^2+h_0)]\|_{L^2}\|\nabla u^{j+1}\|_{L^6}
\\
=&J_{31}+J_{32}.\end{aligned}
\end{equation}By using Sobolev's embedding theorem in $\mathbb{R}^3$, we deduce that
\begin{equation}
\begin{aligned}
\label{4-6-8}
J_{31}=&\|\nabla^3u^{j+1}\|_{L^3} \|(u^j-1) (u^j+1)^2((u^j)^2+h_0)\|_{L^6}\|\nabla^3 u^{j+1}\|_{L^2}
\\
\lesssim&\|\nabla^3u^{j+1}\|_{L^3}\|\nabla^3 u^{j+1}\|_{L^2}\left[\|u^j+1\|_{L^{\infty}}^2(\|u^j\|_{L^{\infty}}^2+|h_0|)\|u^j-1\|_{L^6}\right.
\\
&
+\|u^j-1\|_{L^{\infty}} \|u^j+1\|_{L^{\infty}} (\|u^j\|_{L^{\infty}}^2+|h_0|)\|u^j+1\|_{L^6}\\&\left.+
\|u^j-1\|_{L^{\infty}} \|u^j+1\|_{L^{\infty}} ^2 \|(u^j)^2+h_0\|_{L^6}\right]
\\
\lesssim&\|\nabla^4u^{j+1}\|_{L^2}^{\frac12}\|\nabla^3 u^{j+1}\|_{L^2}^{\frac32}\|u^j\|_{H^2}^5
\\
\lesssim&\|u^j\|_{H^2}^5(\|\nabla^2u^{j+1}\|_{H^2}^2+\|\nabla^3u^{j+1}\|_{H^2}^2),
\end{aligned}
\end{equation}
and
\begin{equation}
\begin{aligned}
\label{4-6-9}
J_{32}=&\|\nabla^3u^{j+1}\|_{L^3}\|\nabla u^{j+1}\|_{L^6} \|\nabla^2 [(u^j-1) (u^j+1)^2((u^j)^2+h_0)]\|_{L^2}
\\
\lesssim&\|\nabla^3u^{j+1}\|_{L^3}\|\nabla u^{j+1}\|_{L^6}  \|\nabla^2 [(u^j-1) (u^j+1)^2 (u^j)^2 ]\|_{L^2}
\\
&+|h_0|\|\nabla^3u^{j+1}\|_{L^3}\|\nabla u^{j+1}\|_{L^6}  \|\nabla^2 [(u^j-1) (u^j+1)^2  ]\|_{L^2}
\\
\lesssim&\|\nabla^3u^{j+1}\|_{L^2}^{\frac12}\|\nabla^4u^{j+1}\|_{L^2}^{\frac12}\|\nabla^2 u^{j+1}\|_{L^2}  (\|u^j\|_{H^2}^3+\|u^j\|_{H^2}^5)
\\
\lesssim&(\|u^j\|_{H^2}^3+\|u^j\|_{H^2}^5)(\|\nabla^2u^{j+1}\|_{H^2}^2+\|\nabla^3u^{j+1}\|_{H^2}^2),
\end{aligned}
\end{equation}
Plugging (\ref{4-6-8}) and (\ref{4-6-9}) into (\ref{4-6-7}), for the case $l=2$, we obtain
\begin{equation}\label{4-6-10}
J_3\lesssim(\|u^j\|_{H^2}^3+\|u^j\|_{H^2}^5)(\|\nabla^2u^{j+1}\|_{H^2}^2+\|\nabla^3u^{j+1}\|_{H^2}^2) .
\end{equation}
Similarly, if $l=2$, we can also obtain
\begin{equation}
\label{4-6-11}
J_4+J_5\lesssim (\|u^j\|_{H^2}^3+\|u^j\|_{H^2}^5)(\|\nabla^2u^{j+1}\|_{H^2}^2+\|\nabla^3u^{j+1}\|_{H^2}^2).
\end{equation}
 Now, summing up the estimates (\ref{4-4}), (\ref{4-5}), (\ref{4-5-1}), (\ref{4-6}), (\ref{4-6-4}), (\ref{4-6-5}), (\ref{4-6-6}), (\ref{4-6-10}) and (\ref{4-6-11}), we arrive at
\begin{equation}
\label{4-8x}\begin{aligned}&
\frac12\frac d{dt} \|u^{j+1}\|_{H^2}^2
 +\delta\|\nabla^3 u^{j+1}\|_{H^2}^2+\kappa\|\nabla^2u^{j+1}\|_{H^2}^2
\\
\leq&C(\|u^j\|_{H^2}^2+\|u^j\|_{H^2}^5 )(\|\nabla^2u^{j+1}\|_{H^2}^2+\|\nabla^3u^{j+1}\|_{H^2}^2).
\end{aligned}
\end{equation}
By taking time integration, we have
\begin{equation}
\begin{aligned}&
\|u^{j+1}\|_{H^2}^2 +\int_0^t(\delta\|\nabla^3 u^{j+1}(s)\|_{H^2}^2+\kappa\|\nabla^2u^{j+1}(s)\|_{H^2}^2 )ds
\\
\leq&C \|u_0\|_{H^2}^2 +C\int_0^t(\|u^j(s)\|_{H^2}^2+\|u^j(s)\|_{H^2}^5 )
 (\|\nabla^3 u^{j+1}(s)\|_{H^2}^2+\|\nabla^2 u^{j+1}(s)\|_{H^2}^2)ds,
\end{aligned}\nonumber
\end{equation}
which from the inductive assumption implies
\begin{equation}
\begin{aligned}&\|u^{j+1}\|_{H^2}^2+ \int_0^t \|\nabla^3 u^{j+1}(s)\|_{H^2}^2+ \|\nabla^2u^{j+1}(s)\|_{H^2}^2 )ds
\\
\leq&C\varepsilon^2+C(M_1^2+M_1^5)\int_0^t(\|\nabla^3 u^{j+1}(s)\|_{H^2}^2+\|\nabla^2 u^{j+1}(s)\|_{H^2}^2)ds,
\end{aligned}\nonumber
\end{equation}
for any $0\leq t\leq T_1$. Take suitable small $\varepsilon>0$, $T_1>0$ and $M_1>0$ such that
\begin{equation}
\label{4-9x}\begin{aligned}&\|u^{j+1}\|_{H^2}^2+ \int_0^t \|\nabla^3 u^{j+1}(s)\|_{H^2}^2+ \|\nabla^2u^{j+1}(s)\|_{H^2}^2 )ds\leq M_1^2,\end{aligned}
\end{equation} for any $t\in[0,T_1]$.
Therefore, (\ref{4-2}) is true for $j+1$ if so for $j$, which implies (\ref{4-2}) is proved for all $j\geq0$.

Next, by using (\ref{4-8x}), we deduce that
\begin{equation}
\begin{aligned}&
\left|\|\mathcal{A}^{j+1}(t)\|_{H^2}^2-\|\mathcal{A}^{j+1}(s)\|_{H^2}^2\right|
\\
=&\left|\int_s^t\frac d{d\tau}\|\mathcal{A}^{j+1}(\tau)\|_{H^2}^2d\tau\right|
\\
\leq&C\int_s^t(\|\mathcal{A}^j(s)\|_{H^2}^2+\|\mathcal{A}^j(s)\|_{H^2}^5)(\|\nabla^3 \mathcal{A}^{j+1}(s)\|_{H^2}^2+\|\nabla^2\mathcal{A}^{j+1}(s)\|_{H^2}^2) ds
\\
\leq&C(M_1^2+M_1^5)\int_s^t(\|\nabla^3 \mathcal{A}^{j+1}(s)\|_{H^2}^2+\|\nabla^2\mathcal{A}^{j+1}(s)\|_{H^2}^2 )ds,\nonumber
\end{aligned}
\end{equation}
for any $t\in[0,T_1]$. Therefore,   due to (\ref{4-9x}), the time integral in the last inequality is finite, and hence $\|\mathcal{A}^{j+1}(t)\|_{H^2}^2$ is continuous  in $t$ for each $j\geq1$. On the other hand, we also need to consider the convergence of the sequence $(\mathcal{A}^j)_{j\geq0}$. Taking the difference of (\ref{4-1})$_1$ for $j$ and $j-1$, it yields that
\begin{equation} \label{4-10x}
 \begin{aligned}
&\partial_t(u^{j+1}-u^j)+ \delta\nabla\Delta (u^{j+1}-u^j)+\kappa_0\Delta (u^{j+1}-u^j)
\\
=&-\kappa_1\Delta\left\{(u^{j}+\kappa_2)(u^j-\kappa_2)(\Delta u^{j+1}-\Delta u^j)\right.\\
&\left.+[(u^{j}+\kappa_2)(u^j-\kappa_2)-(u^{j-1}+\kappa_2)(u^{j-1}-\kappa_2)]\Delta u^j\right\}
\\
&-\kappa_1\Delta\left[u^j\nabla u^j\cdot(\nabla u^{j+1}-\nabla u^j)+(u^j\nabla u^j-u^{j-1}\nabla u^{j-1})\cdot\nabla u^j\right]
\\&+\Delta\left\{(u^j-1)^2(u^j+1) [(u^j)^2+h_0](u^{j+1}-u^j) \right.
\\&\left.-\left[(u^j-1)^2(u^j+1) [(u^j)^2+h_0]-(u^{j-1}-1)^2(u^{j-1}+1) [(u^{j-1})^2+h_0]\right](u^{j}+1) \right\}  ,
                          \end{aligned}
 \end{equation}
Appealing to the same energy estimate as before, we get
\begin{equation}
\label{4-11x}
\begin{aligned}&
\frac12\frac d{dt} \|u^{j+1}-u^j\|_{H^2}^2 +\delta\|\nabla\Delta u^{j+1}-\nabla \Delta u^j\|_{H^2}^2+\kappa\| \Delta u^{j+1}-  \Delta u^j\|_{H^2}^2
\\
\leq&C(\|u^j-u^{j-1}\|_{H^2}^2+\|u^j-u^{j-1}\|_{H^2}^5)( \ \nabla \Delta u^j\|_{H^2}^2+ \|  \Delta u^j\|_{H^2}^2 )
 \\
&+C(\|u^j\|_{H^2}^2 +\|u^j\|_{H^2}^5)( \|\nabla\Delta u^{j+1}-\nabla \Delta u^j\|_{H^2}^2+ \| \Delta u^{j+1}-  \Delta u^j\|_{H^2}^2 ) ,
\end{aligned}
\end{equation}
which is equivalent to
\begin{equation}
\label{4-12}
\begin{aligned}&
\frac12\frac d{dt} \|\mathcal{A}^{j+1}-\mathcal{A}^j\|_{H^2}^2  +\|\nabla^3\mathcal{A}^{j+1}-\nabla^3\mathcal{A}^j\|_{H^2}^2+\|\nabla^2\mathcal{A}^{j+1}-\nabla^2
\mathcal{A}^j\|_{H^2}^2
\\
 \leq&C \|\mathcal{A}^j-\mathcal{A}^{j-1}\|_{H^2}^2(\|\mathcal{A}^j\|_{H^2}^3+\|\mathcal{A}^{j-1}\|_{H^2}^3+1)(\|
 \nabla^3\mathcal{A}^j\|_{H^2}^2+\| \nabla^2
\mathcal{A}^j\|_{H^2}^2) \\
&+C(\|\mathcal{A}^j\|_{H^2}^2+\|\mathcal{A}^j\|_{H^2}^5)(\|\nabla^3\mathcal{A}^{j+1}-\nabla^3\mathcal{A}^j\|_{H^2}^2+\|\nabla^2\mathcal{A}^{j+1}-\nabla^2
\mathcal{A}^j\|_{H^2}^2),
\end{aligned}
\end{equation}
Based on (\ref{4-9x}), by taking time integration, it holds that
\begin{equation}
\begin{aligned}
\nonumber
&\|(\mathcal{A}^{j+1}-\mathcal{A}^j)(t)\|_{H^2}^2 +\int_0^t(\|\nabla^3\mathcal{A}^{j+1}-\nabla^3\mathcal{A}^j\|_{H^2}^2+\|\nabla^2\mathcal{A}^{j+1}-\nabla^2
\mathcal{A}^j\|_{H^2}^2)ds\\
\leq&C(M_1^2+M_1^5)\sup_{0\leq s\leq T_1} \| \mathcal{A}^{j }-\mathcal{A}^{j-1} \|_{H^2}^2
\\
&+C(M_1^2+M_1^5)\int_0^t(\|\nabla^3\mathcal{A}^{j+1}-\nabla^3\mathcal{A}^j\|_{H^2}^2+\|\nabla^2\mathcal{A}^{j+1}-\nabla^2
\mathcal{A}^j\|_{H^2}^2)ds.
\end{aligned}
\end{equation}Since $M_1$ is sufficiently small,  there exists a constant $\lambda\in(0,1)$ such that
\begin{equation}
\label{4-13x}
\sup_{0\leq t\leq T_1}\|\mathcal{A}^{j+1}(t)-\mathcal{A}^j(t)\|_{H^2}^2\leq\lambda\sup_{0\leq t\leq T_1}\|\mathcal{A}^{j}(t)-\mathcal{A}^{j-1}(t)\|_{H^2}^2,
\end{equation}
for any $j\geq1$. Hence, $(\mathcal{A}^j)_{j\geq0}$ is a Cauchy sequence in the Banach space $C([0,T_1];H^2)$, the limit function
$$
\mathcal{A}(t)=\mathcal{A}_0+\lim_{n\rightarrow\infty}\sum_{j=0}^n(\mathcal{A}^{j+1}-\mathcal{A}^j)
$$
exists in $C([0,T_1];H^2)$, satisfies
$$
\sup_{0\leq t\leq T_1}\|\mathcal{A}(t)\|_{H^2}^2\leq \sup_{0\leq t\leq T_1}\liminf_{j\rightarrow\infty}\|\mathcal{A}(t)\|_{H^2}\leq M_1.
$$
Hence, (\ref{4-3}) is proved. Finally, suppose that $\mathcal{A}(t)$ and $\tilde{\mathcal{A}}(t)$ are two solutions in $C([0,T_1];H^2)$ satisfying (\ref{4-3}). By using the same process as in ( \ref{4-13x}) to prove the convergence of $(\mathcal{A}^j)_{j\geq0}$, we find that
$$
\sup_{0\leq t\leq T_1}\|\mathcal{A} (t)-\tilde{\mathcal{A}}(t)\|_{H^2}^2\leq\lambda\sup_{0\leq t\leq T_1}\|\mathcal{A} (t)-\tilde{\mathcal{A}}(t)\|_{H^2}^2,
$$for $\lambda\in(0,1)$, which implies that $\mathcal{A}(t)=\tilde{\mathcal{A}}(t)$.
The proof of uniqueness is complete and thus the proof of Lemma \ref{lem4.1} is complete too.

\end{proof}

\subsection{Global well-posedness}

In this subsection, we shall combine all the energy estimates that we have derived in the previous sections and the Sobolev interpolation to prove Theorem \ref{thm1.1}.

\begin{proof}[Proof of Theorem \ref{thm1.1}]
We first close the energy estimates at each $l$-th level in our weak sense to prove (\ref{1-3}). Let $N\geq1$ and $0\leq l\leq m-1$ with $1\leq m\leq N$.
Summing up the estimates (\ref{3-3}) of Lemma \ref{lem3.1} from $k=l$ to $m$, we easily obtain
\begin{equation}
\begin{aligned}\label{6-1}&
\frac d{dt}\sum_{l\leq k\leq m} \|\nabla^ku\|_{L^2}^2 + \sum_{l\leq k\leq m}(\|\nabla^{k+3}u\|_{L^2}^2+\|\nabla^{k+2}u\|_{L^2}^2)
\\
\leq& (\varepsilon^2+\varepsilon^5)\sum_{l\leq k\leq m}((\|\nabla^{k+3}u\|_{L^2}^2+\|\nabla^{k+2}u\|_{L^2}^2 ).
\end{aligned}\end{equation}
Since $\varepsilon>0$ is small, we deduce that there exists a constant $C>0$ such that for $0\leq l\leq m-1$,
\begin{equation}
\begin{aligned}\label{6-2}
\frac d {dt}\sum_{l\leq k\leq m}\|\nabla^ku\|_{L^2}^2 +C_0 \sum_{l\leq k\leq m}(\|\nabla^{k+3}u\|_{L^2}^2+\|\nabla^{k+2}u\|_{L^2}^2)
 \leq 0.
\end{aligned}\end{equation}
Define $\mathcal{E}_l^m(t)$ to be $\frac1{C_0}$ times the expression under the time derivative in (\ref{6-2}). Hence, we may write (\ref{6-2}) as that for $0\leq l\leq m-1$,
\begin{equation}
\label{6-3}
\frac d{dt}\mathcal{E}_l^m(t) +\|\nabla^{l+3}u\|_{H^{m-l}}^2+\|\nabla^{l+2}u\|_{H^{m-l}}^2
\lesssim 0.
\end{equation}
Taking $l=0$ and $m=3$ in (\ref{6-3}) and integrating directly in time, we deduce that
\begin{equation}
\label{6-4}
\|u(t)\|_{H^2}^2 \lesssim \mathcal{E}_0^2(0)\lesssim\|u_0\|_{H^2}^2 .
\end{equation}
 By a standard continuity argument, this closes the a priori estimates (\ref{1-2}) if at the initial time $\|u_0\|_{H^2}^2 $ is sufficiently small. This in turn allows us to take $l=0$ and $m=N$ in (\ref{6-4}), and then integrate it directly in time to obtain (\ref{1-3}). Hence,   the proof of Theorem \ref{thm1.1} is complete.
\end{proof}

\section{Proof of Theorem \ref{thm1.2}}
\subsection{Negative Sobolev estimates}

In this subsection, we  derive the evolution of the negative Sobolev norms of the solution.
\begin{lemma}
\label{lem5-1}
Suppose that all assumptions in Theorem \ref{thm1.1} hold. Then, for $s\in[0,\frac12]$, we have
\begin{equation}
\label{5-1}
\frac d{dt}\|\Lambda^{-s}u\|_{L^2}^2+\|\Lambda^{-s}\nabla\Delta u\|_{L^2}^2+\|\Lambda^{-s}\Delta u\|_{L^2}^2\lesssim \|\nabla^2u\|_{H^2}^2\|\Lambda^{-s}u\|_{L^2}.
\end{equation}
\end{lemma}
\begin{proof}
Applying $\Lambda^{-s}$ to (\ref{1-1}), multiplying the resulting identity by $\Lambda^{-s}u$, integrating over $\mathbb{R}^3$ by parts, we deduce that
\begin{equation}
\begin{aligned}\label{5-2}&
\frac12\frac d{dt}\|\Lambda^{-s}u\|_{L^2}^2+\delta\|\Lambda^{-s}\nabla\Delta u\|_{L^2}^2+\kappa_0\|\Lambda^{-s}\Delta u\|_{L^2}^2
\\
=&-\kappa_1\int_{\mathbb{R}^3}\Lambda^{-s}\left\{\Delta\left[\left(u+ \kappa_2\right)\left(u-\kappa_2\right)\Delta u\right]\right\}\cdot\Lambda^{-s}udx
 -\kappa_1\int_{\mathbb{R}^3}\Lambda^{-s}\left[\Delta\left(u|\nabla u|^2\right)\right]\cdot\Lambda^{-s}udx
\\&+\int_{\mathbb{R}^3}\Lambda^{-s}\Delta\left[(u-1)^2(u+1)^2(u^2+h_0)\right]\cdot\Lambda^{-s}udx
\\
=&K_1+K_2+K_3
.
\end{aligned}
\end{equation}
We will estimate the three terms of the right hand side of (\ref{5-2}) one by one. For the first term, we have
\begin{equation}
\begin{aligned}
\label{5-3}
K_1=&-\kappa_1\int_{\mathbb{R}^3}\Lambda^{-s}\left\{\Delta\left[\left(u+ \kappa_2\right)\left(u-\kappa_2\right)\Delta u\right]\right\}\cdot\Lambda^{-s}udx
\\
\lesssim&\int_{\mathbb{R}^3}|\Lambda^{-s}(|\nabla u|^2\Delta u)||\Lambda^{-s}u|dx+\int_{\mathbb{R}^3}|\Lambda^{-s}(u|\Delta u|^2)||\Lambda^{-s}u|dx
\\
&+
\int_{\mathbb{R}^3}|\Lambda^{-s}(u\nabla u\nabla\Delta u)||\Lambda^{-s}u|dx+\int_{\mathbb{R}^3}|\Lambda^{-s}((u+\kappa_2)(u-\kappa_2)\Delta^2 u)||\Lambda^{-s}u|dx
\\
\lesssim&\|\Lambda^{-s}u\|_{L^2} \|\Lambda^{-s}(|\nabla u|^2\Delta u)\|_{L^2}+\|\Lambda^{-s}u\|_{L^2} \| \Lambda^{-s}(u|\Delta u|^2)\|_{L^2}
\\
&
+\|\Lambda^{-s}u\|_{L^2} \| \Lambda^{-s}(u\nabla u\nabla\Delta u)\|_{L^2}+\|\Lambda^{-s}u\|_{L^2} \| \Lambda^{-s}((u+\kappa_2)(u-\kappa_2)\Delta^2 u)\|_{L^2}
\\
=&K_{11}+K_{12}+K_{13}+K_{14}.
\end{aligned}
\end{equation}
We estimate $K_{11}$ as
\begin{equation}
\begin{aligned}\label{5-4}
K_{11}=&\|\Lambda^{-s}u\|_{L^2} \|\Lambda^{-s}(|\nabla u|^2\Delta u)\|_{L^2}
\\
\lesssim&\|\Lambda^{-s}u\|_{L^2} \| |\nabla u|^2\Delta u \|_{L^{\frac1{\frac12+\frac s3}}}
\\
\lesssim&\|\Lambda^{-s}u\|_{L^2} \|\nabla u\|_{L^{\infty}}\|\nabla u\|_{L^{\frac3s}}\|\Delta u\|_{L^2}
\\
\lesssim&\|\Lambda^{-s}u\|_{L^2} \|\nabla^2u\|_{L^2}^{\frac12}\|\nabla^3u\|_{L^2}^{\frac12}\|\nabla^2u\|_{L^2}^{\frac12+s}\|\nabla^3u\|_{L^2}^{\frac12-s}\|\nabla^2u\|_{L^2}
\\
\lesssim&\|\nabla^2u\|_{L^2}\|\Lambda^{-s}u\|_{L^2} \|\nabla^2u\|_{L^2}^{1+s}\|\nabla^3u\|_{L^2}^{1-s}
\\
\lesssim&\|\Lambda^{-s}u\|_{L^2}(\|\nabla^2u\|_{L^2}^{2}+\|\nabla^3u\|_{L^2}^{2}),
\end{aligned}
\end{equation}
where we have used the fact $\|\nabla^2u\|_{L^2}\leq C$, which was proved in Theorem \ref{thm1.1}. $K_{12}$ can be estimated as
\begin{equation}
\begin{aligned}\label{5-5}
K_{12}=&\|\Lambda^{-s}u\|_{L^2} \| \Lambda^{-s}(u|\Delta u|^2)\|_{L^2}
\\
\lesssim&\|\Lambda^{-s}u\|_{L^2} \| u|\Delta u|^2 \|_{L^{\frac1{\frac12+\frac s3}}}
\\
\lesssim&\|\Lambda^{-s}u\|_{L^2} \|  u\|_{L^{\infty}}\|\Delta u\|_{L^{\frac3s}}\|\Delta u\|_{L^2}
\\
\lesssim&\|\Lambda^{-s}u\|_{L^2} (\|\nabla u\|_{L^2}^{\frac12}\|\nabla^2u\|_{L^2}^{\frac12})\|\nabla^3u\|_{L^2}^{\frac12+s}\|\nabla^4u\|_{L^2}^{\frac12-s}\|\nabla^2u\|_{L^2}
\\
\lesssim&\|\nabla u\|_{H^1}\|\Lambda^{-s}u\|_{L^2} \|\nabla^3u\|_{L^2}^{\frac12+s}\|\nabla^4u\|_{L^2}^{\frac12-s}\|\nabla^2u\|_{L^2}
\\
\lesssim&\|\Lambda^{-s}u\|_{L^2}(\|\nabla^2u\|_{L^2}^{2}+\|\nabla^3u\|_{L^2}^{2}+\|\nabla^4u\|_{L^2}^{2}),
\end{aligned}
\end{equation}
where we have used the fact $\|\nabla u\|_{H^1}\leq C$, which was proved in Theorem \ref{thm1.1}. We also have
\begin{equation}
\begin{aligned}\label{5-6}
K_{13}=&\|\Lambda^{-s}u\|_{L^2} \| \Lambda^{-s}(u\nabla u\nabla\Delta u)\|_{L^2}
\\
\lesssim&\|\Lambda^{-s}u\|_{L^2} \|u\nabla u\nabla\Delta u\|_{L^{\frac1{\frac12+\frac s3}}}
\\
\lesssim&\|\Lambda^{-s}u\|_{L^2} \|  u\|_{L^{\infty}}\|\nabla u\|_{L^{\frac3s}}\|\nabla\Delta u\|_{L^2}
\\
\lesssim&\|\Lambda^{-s}u\|_{L^2} (\|\nabla u\|_{L^2}^{\frac12}\|\nabla^2u\|_{L^2}^{\frac12})\|\nabla^2u\|_{L^2}^{\frac12+s}\|\nabla^3u\|_{L^2}^{\frac12-s}\|\nabla^3u\|_{L^2}
\\
\lesssim&\|\nabla u\|_{H^1}\|\Lambda^{-s}u\|_{L^2} \|\nabla^2u\|_{L^2}^{\frac12+s}\|\nabla^3u\|_{L^2}^{\frac32-s}
\\
\lesssim&\|\Lambda^{-s}u\|_{L^2}(\|\nabla^2u\|_{L^2}^{2}+\|\nabla^3u\|_{L^2}^{2} ),
\end{aligned}
\end{equation}
and
\begin{equation}
\begin{aligned}\label{5-7}
K_{14}=&\|\Lambda^{-s}u\|_{L^2} \| \Lambda^{-s}((u+\kappa_2)(u-\kappa_2)\Delta^2 u)\|_{L^2}
\\
\lesssim&\|\Lambda^{-s}u\|_{L^2} \|(u+\kappa_2)(u-\kappa_2)\Delta^2 u\|_{L^{\frac1{\frac12+\frac s3}}}
\\
\lesssim&\|\Lambda^{-s}u\|_{L^2} \|  u+\kappa_2\|_{L^{\infty}}\|   u-\kappa_2\|_{L^{\frac3s}}\| \Delta^2 u\|_{L^2}
\\
\lesssim&\|\Lambda^{-s}u\|_{L^2}  \|\nabla u\|_{L^2}^{\frac12}\|\nabla^2u\|_{L^2}^{\frac12} \|\nabla u\|_{L^2}^{\frac12+s}\|\nabla^2u\|_{L^2}^{\frac12-s}\|\nabla^4u\|_{L^2}
\\
=&\|\Lambda^{-s}u\|_{L^2}  \|\nabla u\|_{L^2}^{1+s}\|\nabla^2u\|_{L^2}^{1-s} \|\nabla^4u\|_{L^2}
\\
\lesssim&\|\Lambda^{-s}u\|_{L^2}  \| u\|_{L^2}^{\frac12+\frac s2}\| \nabla^2u\|_{L^2}^{\frac12+\frac s2}\|\nabla^2u\|_{L^2}^{1-s} \|\nabla^4u\|_{L^2}
\\
\lesssim&\| u\|_{L^2}^{\frac12+\frac s2}\|\nabla^2u\|_{L^2}^{\frac12-\frac s2}  \|\Lambda^{-s}u\|_{L^2} \|\nabla^2u\|_{L^2}  \|\nabla^4u\|_{L^2}
\\
\lesssim&\|\Lambda^{-s}u\|_{L^2}(\|\nabla^2u\|_{L^2}^{2}+\|\nabla^4u\|_{L^2}^{2} ),
\end{aligned}
\end{equation}
where we have used the fact $\| u\|_{H^2}\leq C$, which was proved in Theorem \ref{thm1.1}.  Combining (\ref{5-3})-(\ref{5-7}) together gives
\begin{equation}
\label{5-8}
K_1\lesssim\|\Lambda^{-s}u\|_{L^2}(\|\nabla^2u\|_{L^2}^{2}+\|\nabla^3u\|_{L^2}^{2}+\|\nabla^4u\|_{L^2}^{2}).
\end{equation}
Next, on the basis of the estimates (\ref{5-4})-(\ref{5-6}), we estimate $K_2$ as
\begin{equation}
\begin{aligned}\label{5-9}
K_2=&-\kappa_1\int_{\mathbb{R}^3}\Lambda^{-s}\left[\Delta\left(u|\nabla u|^2\right)\right]\cdot\Lambda^{-s}udx
\\
\lesssim& \int_{\mathbb{R}^3}\Lambda^{-s}(|\nabla u|^2\Delta u)\cdot\Lambda^{-s}udx +\int_{\mathbb{R}^3}\Lambda^{-s}(u|\Delta u|^2)\cdot\Lambda^{-s}udx+
\int_{\mathbb{R}^3}\Lambda^{-s}(u\nabla u\nabla\Delta u)\cdot\Lambda^{-s}udx
\\
\lesssim&\|\Lambda^{-s}u\|_{L^2}(\|\nabla^2u\|_{L^2}^{2}+\|\nabla^3u\|_{L^2}^{2}+\|\nabla^4u\|_{L^2}^{2}).
\end{aligned}
\end{equation}
For $K_3$, we have
\begin{equation}
\begin{aligned}\label{5-10}
K_3=&\int_{\mathbb{R}^3}\Lambda^{-s}\Delta\left[(u-1)^2(u+1)^2(u^2+h_0)\right]\cdot\Lambda^{-s}udx
\\
\lesssim&\int_{\mathbb{R}^3}|\Lambda^{-s} \left[ (u+1)^2(u^2+h_0)|\nabla u|^2\right]||\Lambda^{-s}u|dx
\\
&+
\int_{\mathbb{R}^3}|\Lambda^{-s} \left[ (u-1)(u+1) (u^2+h_0)|\nabla u|^2\right]||\Lambda^{-s}u|dx
\\
&+
\int_{\mathbb{R}^3}|\Lambda^{-s} \left[ (u-1)(u+1)^2u|\nabla u|^2\right]||\Lambda^{-s}u|dx
\\
&+
\int_{\mathbb{R}^3}|\Lambda^{-s} \left[ (u-1)(u+1)^2(u^2+h_0) \Delta u \right]||\Lambda^{-s}u|dx
\\
&+
\int_{\mathbb{R}^3}|\Lambda^{-s} \left[ (u-1)^2(u+1)(u^2+h_0) \Delta u \right]||\Lambda^{-s}u|dx
\\
&+
\int_{\mathbb{R}^3}|\Lambda^{-s} \left[ (u-1)^2(u^2+h_0)|\nabla u|^2\right]||\Lambda^{-s}u|dx
\\
&+
\int_{\mathbb{R}^3}|\Lambda^{-s} \left[ (u-1)^2(u+1) u|\nabla u|^2\right]||\Lambda^{-s}u|dx
\\
&+
\int_{\mathbb{R}^3}|\Lambda^{-s} \left[ (u+1)^2(u-1)^2|\nabla u|^2\right]||\Lambda^{-s}u|dx
\\
&+
\int_{\mathbb{R}^3}|\Lambda^{-s} \left[u (u+1)^2(u-1)^2  \Delta u \right]||\Lambda^{-s}u|dx
\\
=&K_{31}+K_{32}+K_{33}+K_{34}+K_{35}+K_{36}+K_{37}+K_{38}+K_{39}.
\end{aligned}
\end{equation}
Note that
\begin{equation}
\begin{aligned}
\label{5-11}
K_{31}=&\int_{\mathbb{R}^3}|\Lambda^{-s} \left[ (u+1)^2(u^2+h_0)|\nabla u|^2\right]||\Lambda^{-s}u|dx
\\
\lesssim&\|\Lambda^{-s}u\|_{L^2}\|\Lambda^{-s} \left[ (u+1)^2(u^2+h_0)|\nabla u|^2\right]\|_{L^2}
\\
\lesssim&\|\Lambda^{-s}u\|_{L^2}\|  (u+1)^2(u^2+h_0)|\nabla u|^2 \|_{L^{\frac1{\frac12+\frac s3}}}
\\
\lesssim&\|\Lambda^{-s}u\|_{L^2}\|u^2+h_0\|_{L^{\infty}}\|u+1\|_{L^{\frac 3s}}\|u+1\|_{L^6}\|\nabla u\|_{L^6}^2
\\
\lesssim&\|\Lambda^{-s}u\|_{L^2}(\|\nabla u\|_{L^2}\|\nabla^2u\|_{L^2}+|h_0|) \|\nabla u\|_{L^2}^{\frac12+s} \|\nabla^2u\|_{L^2}^{\frac12-s} \|\nabla u\|_{L^2}\|\nabla^2 u\|_{L^2}^2
\\
\lesssim&\|\Lambda^{-s}u\|_{L^2}\|\nabla^2 u\|_{L^2}^2,
\end{aligned}
\end{equation}
where we have used the fact $\| u\|_{H^2}\leq C$, which was proved in Theorem \ref{thm1.1}. Similarly, we obtain
\begin{equation}
\label{5-12}
K_{32}+K_{33}+K_{36}+K_{37}+K_{38}\lesssim \|\Lambda^{-s}u\|_{L^2}\|\nabla^2 u\|_{L^2}^2.
\end{equation}
For $K_{34}$, we can estimate as
\begin{equation}
\label{5-13}
\begin{aligned}
K_{34}=&\int_{\mathbb{R}^3}|\Lambda^{-s} \left[ (u-1)(u+1)^2(u^2+h_0) \Delta u \right]||\Lambda^{-s}u|dx
\\
\lesssim&\|\Lambda^{-s}u\|_{L^2}\|\Lambda^{-s} \left[ (u-1)(u+1)^2(u^2+h_0) \Delta u \right]\|_{L^2}
\\
\lesssim&\|\Lambda^{-s}u\|_{L^2}\|  (u-1)(u+1)^2(u^2+h_0) \Delta u  \|_{L^{\frac1{\frac12+\frac s3}}}
\\
\lesssim&\|\Lambda^{-s}u\|_{L^2}\|u^2+h_0\|_{L^{\infty}}\|u-1\|_{L^{\frac3s}}\|u+1\|_{L^{\infty}}^2\|\Delta u\|_{L^2}
\\
\lesssim&\|\Lambda^{-s}u\|_{L^2}(\|\nabla u\|_{L^2}\|\nabla^2u\|_{L^2}+|h_0|) \|\nabla u\|_{L^2}^{\frac12+s} \|\nabla^2u\|_{L^2}^{\frac12-s}
(\|\nabla u\|_{L^2}^{\frac12}\|\nabla^2u\|_{L^2}^{\frac12})^2\|\nabla^2u\|_{L^2}
\\
\lesssim&\|\Lambda^{-s}u\|_{L^2}\|\nabla^2 u\|_{L^2}^2,
\end{aligned}
\end{equation}
where we have used the fact $\| u\|_{H^2}\leq C$, which was proved in Theorem \ref{thm1.1}. Similarly, we obtain
\begin{equation}
\label{5-14}
K_{35}+K_{39}\lesssim \|\Lambda^{-s}u\|_{L^2}\|\nabla^2 u\|_{L^2}^2.
\end{equation}
It then follows from (\ref{5-10})-(\ref{5-14}) that
\begin{equation}
\label{5-15}
K_3\lesssim\|\Lambda^{-s}u\|_{L^2}\|\nabla^2 u\|_{L^2}^2.
\end{equation}
Plugging (\ref{5-8}), (\ref{5-9}) and (\ref{5-15}) into (\ref{5-2}), we deduce (\ref{5-1}).

\end{proof}
\subsection{Decay estimates}

In the following, we prove Theorem \ref{thm1.2} for $s\in[0,\frac12]$.
\begin{proof}[Proof of Theorem \ref{thm1.2}] Define
$$
\mathcal{E}_{-s}(t):=\|\Lambda^{-s}u(t)\|_{L^2}^2 .
$$
For inequality (\ref{5-1}), integrating in time, by the bound (\ref{1-3}), we have
\begin{equation}
\label{6-5c}
\begin{aligned}
\mathcal{E}_{-s}(t)\leq&\mathcal{E}_{-s}(0)+C\int_0^t\|\nabla^2u\|_{H^2}^2\sqrt{\mathcal{E}_{-s}(\tau)}d\tau
\\
\leq&C_0\left(1+\sup_{0\leq\tau\leq t}\sqrt{\mathcal{E}_{-s}(\tau)}d\tau\right),
\end{aligned}
\end{equation}
which implies (\ref{1-4}) for $s\in[0,\frac12]$, that is
\begin{equation}
\label{6-5}
\|\Lambda^{-s}u(t)\|_{L^2}^2\leq C_0.
\end{equation}
Moreover, if  $l=1,2,\cdots,N-1$, we may use Lemma \ref{lem2.2} to have
 $$
\|\nabla^{l+2}f\|_{L^2}\geq C\|\Lambda^{-s}f\|_{L^2}^{-\frac2{l+s}}\|\nabla^lf\|_{L^2}^{1+\frac2{l+s}}.
$$
Then, by this facts and (\ref{6-5}), we get
\begin{equation}\label{7-1}
\|\nabla^{l+2}u\|_{L^2}^2\geq C_0(\|\nabla^{l}u\|_{L^2}^2)^{1+\frac2{k+s}}.
\end{equation}
 Hence, for $1=1,2,\cdots,N-1$,
$$
\|\nabla^{l+2} u \|_{H^{N-l }}^2\geq C_0(\|\nabla^{l} u \|_{H^{N-l}}^2)^{1+\frac2{l+s}}.
$$
Thus, we deduce from (\ref{6-3}) with $m=N$ the following inequality
\begin{equation}
\label{7-2}
\frac d{dt}\mathcal{E}_l^N+C_0\left(\mathcal{E}_l^N\right)^{1+\frac2{l+s}}\leq 0,\quad\hbox{for}~l=1,2,\cdots,N-1.
\end{equation}
Solving this inequality directly gives
\begin{equation}
\label{7-3}
\mathcal{E}_l^N(t)\leq C_0(1+ t)^{-\frac12(l+s)},\quad\hbox{for}~l=1,2,\cdots,N-1,
\end{equation}
which means (\ref{1-5}) holds. Hence, we   complete the proof of Theorem \ref{thm1.2}.

\end{proof}

\section*{Acknowledgement}
The author  appreciate very much the useful suggestions of Prof. Hao Wu. His suggestions will benefit the improvement of the paper and my future research.

}

\end{document}